\documentclass[11pt]{article}
\usepackage[utf8]{inputenc}  
\usepackage[T1]{fontenc}
\usepackage[english]{babel}
\usepackage{amsmath,textcomp,amssymb,geometry,graphicx,enumerate}
\usepackage{amsmath}
\usepackage{algorithm} 
\usepackage[noend]{algpseudocode} 
\usepackage{hyperref}
\usepackage{stmaryrd}
\usepackage{tabularx}
\hypersetup{
    colorlinks=true,
    linkcolor=blue,
    filecolor=magenta,      
    urlcolor=blue,
}

\def\Name{Chen Song}  
\def\Login{University of Illinois at Chicago}
\def\Session{}
\newtheorem{lemma}{Lemma}[section]
\newtheorem{theorem}[lemma]{Theorem}

\newtheorem{proposition}[lemma]{Proposition}

\newtheorem{definition}[lemma]{Definition}
\newtheorem{construction}[lemma]{Construction}

\usepackage{amsmath}
\newcommand{\leqpar}{\underset{{\scriptscriptstyle (}-{\scriptscriptstyle )}}{<}}

\title{Stability of Kernel Bundles}
\author{\Name, \texttt{\Login}}
\date{}

\def\endproofmark{$\Box$}
\newenvironment{proof}{\par{\bf Proof}:}{\endproofmark\smallskip}

\begin{document}
\maketitle

\begin{abstract}
    In this paper, we study the stability of general kernel bundles on $\mathbb{P}^n$. Let $a,b,d>0$ be integers. A kernel bundle $E_{a,b}$ on $\mathbb{P}^n$ is defined as the kernel of a surjective map $\phi:\mathcal{O}_{\mathbb{P}^n}(-d)^a\rightarrow \mathcal{O}_{\mathbb{P}^n}^b$.
Here $\phi$ is represented by a $b\times a$ matrix $(f_{ij})$ where the entries $f_{ij}$ are polynomials of degree $d$. We give sufficient conditions for semistability of a general kernel bundle on $\mathbb{P}^n$, in terms of its Chern class. 
\end{abstract}

\section{Introduction}
In this paper, we study the stability of kernel bundles on projective space $\mathbb{P}^n$. Let $a,b,d>0$ be integers. A \textbf{kernel bundle} $E_{a,b}$ on $\mathbb{P}^n$ is defined by the following short exact sequence 
$$0\xrightarrow{}E_{a,b}\xrightarrow{}\mathcal{O}_{\mathbb{P}^n}(-d)^a\xrightarrow[]{\phi}\mathcal{O}_{\mathbb{P}^n}^b\xrightarrow{}0.$$
Here $\phi$ is a surjective map represented by a $b\times a$ matrix $(f_{ij})$, where the entries $f_{ij}$ are polynomials of degree $d$. We give sufficient conditions on the pair $(a,b)$ such that for large enough $d$, a general kernel bundle $E_{a,b}$ is semistable. \par
In the study of vector bundles, stability is a fundamental property which has wide applications.  Let $(X,\mathcal{O}_X(1))$ be an $n$-dimensional polarized variety. Let $\mathcal{F}$ be a torsion free coherent sheaf on $X$.  Let $r(\mathcal{F})$ be the rank of $\mathcal{F}$. The slope of $\mathcal{F}$ is $\mu(\mathcal{F}):=\frac{c_1(\mathcal{F})\cdot \mathcal{O}_X(1)^{n-1}}{r(\mathcal{F})}$. We say $\mathcal{F}$ is (semi)stable if for every coherent subsheaf $\mathcal{W}$ of $\mathcal{F}$ with $0<r(\mathcal{W})<r(\mathcal{F})$, we have $\mu(\mathcal{W})\leqpar \mu(\mathcal{F})$. By Harder-Narasimhan filtration, the semistable bundles are fundamental building blocks of vector bundles. Therefore, looking for semistable bundles is crucial to studying vector bundles on algebraic varieties. Semistable bundles also behave well in families and form projective moduli spaces, see \cite{Mar}. \par
Although kernel bundles has been intensely studied, stability of a general kernel bundles for high degree is still an open problem. In this paper, we study the stability of kernel bundles on $\mathbb{P}^n$ of high degree and prove the following result. 
\begin{theorem}[Main Theorem]
    Let $k=(n+1)^2-\sum_{i=2}^{n+1}( (n+1) \mod i)$. For a given pair of positive integers $(a,b)$, if we can write $a=mb-j$ for some integers $j$, $m$ with $0\leq j \leq b-1$ and $2 \leq m \leq k$, then a general kernel bundle $E_{a,b}$ on $\mathbb{P}^n$ is semistable for $d \gg 0$. 
\end{theorem}
On $\mathbb{P}^n$, kernel bundles of the form $E_{a,1}$ are called \emph{syzygy bundles}.  In \cite[Theorem 6.3]{Br08}, Brenner provides a method to compute the maximal slope of all proper subbundles $\mu_{\max}$ of a syzygy bundle. Based on this, Brenner gives a criterion of stability of syzygy bundles. In \cite[Theorem 3.5]{Co10}, Costa, Macias Marques and Mir\'{o}-Roig prove that on $\mathbb{P}^n$ there exists a stable syzygy bundle $E_{a,1}$ where $\phi$ defined by a family of $a$ polynomials if $n+1\leq a \leq \binom{d+2}{2}+n-2$ and $(n,a,d)\neq (2,5,2)$. In \cite[Theorem 4]{Coa11}, Coand\u{a} proves that for $n\geq 3$ there exists a stable syzygy bundle of form $E_{a,1}$ on $\mathbb{P}^n$ if $n+1\leq a \leq \binom{n+d}{d}$. \par
For kernel bundles of the form $E_{a,b}$ on $\mathbb{P}^n$, in \cite[Theorem 8.1]{Bo90}, Bohnhorst and Spindler gives a criterion of semistability when $a-b=n$.  When $d=1$, the dual of kernel bundle is called Steiner bundle, which has the same stability as kernel bundles. Steiner bundles are first introduced by Dolgachev and Kapranov \cite{Dol}. Stability of exceptional Steiner bundles is studied in \cite{Bra1}, \cite{Bra2} and \cite{Hui}. In a recent result \cite[Theorem 5.1]{Cos} of Coskun, Huizenga and Smith, they prove that on $\mathbb{P}^n$ the kernel bundle $E_{a,b}$ is stable if $d=1$ and it is semistable if $a-b\geq n$ and $\frac{n}{b} \leq \frac{a-b}{b} < \frac{n-1+\sqrt{n^2+2n-3}}{2}$ for arbitrary $d$.  \par

\textbf{Organization of the paper.}  In \ref{pre}, we recall the preliminary facts needed in the rest of the paper, including Brenner's theorem on the maximal slope of syzygy bundles. In \ref{exi}, we give the proof of our main theorem. We construct a syzygy bundle $E_{k,1}$ with a upper bound of $\mu_{\max}(E_{k,1})$. Then we construct a short exact sequence of kernel bundles $0\xrightarrow{}E_{k(b-1)-1,b-1}\xrightarrow{}E_{kb-1,b}\xrightarrow[]{\psi}E_{k,1}\xrightarrow{}0$. By induction on $b$, we use the upper bound of $\mu_{\max}(E_{k,1})$ to find an upper bound of $\mu_{\max}(E_{kb-1,b})$. We use similar short exact sequence and the upper bound of $\mu_{\max}(E_{kb-1,b})$ to find the upper bound of $\mu_{\max}(E_{a,b})$ and shows the stability in our theorem. Finally, we provide a method to prove the stability of $E_{17,2}$, which is not covered by our main theorem. 

\textbf{Acknowledgements.} I extend my deepest gratitude to my advisor, Izzet Coskun, for providing invaluable guidance and insightful suggestions throughout the course of my research on this topic. I would like to thank Yeqin Liu for his constructive remarks on early drafts of this paper. I would also thank Sixuan Lou for many useful discussions. 

\section{Preliminaries} \label{pre}
In this section, we collect necessary preliminaries for the later proof.  First, we recall the definition of stability of a sheaf. 
\begin{definition}
    Let $(X,\mathcal{O}_X(1))$ be an $n$-dimensional polarized projective variety. Let $\mathcal{F}$ be a torsion free coherent sheaf on $X$. The degree of $\mathcal{F}$ is $\text{deg}(\mathcal{F}):=c_1(\mathcal{F})\cdot \mathcal{O}_X(1)^{n-1}$. Let $r(\mathcal{F})$ be the rank of $\mathcal{F}$. The slope of $\mathcal{F}$ is $\mu(\mathcal{F}):=\frac{\text{deg}(\mathcal{F})}{r(\mathcal{F})}$. We say $\mathcal{F}$ is (semi)stable if for every coherent subsheaf $\mathcal{W}$ of $\mathcal{F}$ with $0<r(\mathcal{W})<r(\mathcal{F})$, $\mu(\mathcal{W}) \leqpar \mu(\mathcal{F})$. The maximal slope $\mu_{\max}$ of a sheaf is defined to be the maximum over slopes of all subsheaves.
\end{definition}
\begin{definition}
    Let $X$ be a smooth projective algebraic variety over an algebraically closed field $K$. Let $L$ be a very ample line bundle on $X$. The syzygy bundle $M_L$ associated to $L$ is defined by the kernel of the evaluation map 
$$\phi _L: H^0(X,L) \otimes_K \mathcal{O}_X \longrightarrow L.$$
By this definition, we have a short exact sequence 
$$0\xrightarrow{}M_L\xrightarrow{}H^0(X,L) \otimes_K \mathcal{O}_X\xrightarrow[]{\phi_L} L \xrightarrow{}0.$$
\end{definition}

We need the following result on the maximal slope of syzygy bundles in \cite[Theorem 6.3]{Br08}.

\begin{theorem}
\label{Bre}
Let $f_i$, $i\in I={1,...,n}$, denote a set of primary monomials in  $k[x_0,...,x_n]$ of degree $d_i$. Then the maximal slope of $\emph{Syz}(f_i,i\in I)$ is 
\[\mu_\emph{max}(\emph{Syz}(f_i,i\in I))=\max_{J\subset I,|J|\geq 2}\{ \frac{d_J-\sum_{i\in J}d_i}{|J|-1}\}\]
where $d_J$ is the degree of the highest common factor of $f_i$, $i\in J$. 
\end{theorem}
\

In our case, for syzygy bundle $E_{a,1}$ defined by 
\[0\xrightarrow{}E_{a,1}\xrightarrow{}\mathcal{O}_{\mathbb{P}^n}(-d)^a \xrightarrow[]{\phi}\mathcal{O}_{\mathbb{P}^n}\xrightarrow{}0,\]
if we denote $|J|=r$, we have 
\[\mu_\emph{max}(E_{a,1})=\max_{J\subset I,r\geq 2}\{ \frac{d_J-rd}{r-1}\}.\]

We also need the following result from \cite{Co10} and \cite{Coa11}.  

\begin{theorem}
    Let $P_n(d):=H^0(\mathbb{P}^n,\mathcal{O}_{\mathbb{P}^n}(d))=\frac{(d+1)\cdot ... \cdot (d+n)}{n!}$. Let $n\geq 3$, $d \geq 1$, and $n+1 \leq a \leq P_n(d)$ be integers. Then there is a stable syzygy bundle $E_{a,1}$. 
\end{theorem}

\section{Existence of Semistable Kernel Bundles} \label{exi}

We will use a special syzygy bundle $E_{k,1}$ constructed by the following. The idea is to find an $E_{k,1}$ with the largest possible maximal slope. \par
\begin{construction}
Let $d>0$ be an integer. Let $A$ be a real number. We will construct a syzygy bundle $E_{k,1}$ on $\mathbb{P}^n$ with $k=(n+1)^2-\sum_{i=2}^{n+1}( (n+1) \mod i)$. 
Define $d_r:=(r+A(r-1))d$. Then we have $\frac{d-d_r}{r-1}=-(A+1)d$ and $d_r-d_{r+s}=-s(1+A)d$ for positive integers $r,s$. \par 
For an integer $2 \leq i \leq n$, write $d-\lfloor d_i \rfloor = p_i(i-1)+q_i$ with $0\leq q_i<i-1$. Then we know $\lceil \frac{d-\lfloor d_i \rfloor}{i-1} \rceil =q_i+1$ and $\lfloor \frac{d-\lfloor d_i \rfloor}{i-1} \rfloor =q_i$. \par
Let $E_{k,1}$ to be the syzygy bundle defined by the following $k$ monomials 
\begin{align*}
  &  x_0^d, x_1^d,x_2^d,..., x_n^d, \\
&x_0^{\lfloor d_2 \rfloor}x_1^{d-\lfloor d_2 \rfloor}, 
x_0^{d-\lfloor d_2 \rfloor}x_1^{\lfloor d_2 \rfloor},
x_2^{\lfloor d_2 \rfloor}x_3^{d-\lfloor d_2 \rfloor}, 
x_2^{d-\lfloor d_2 \rfloor}x_3^{\lfloor d_2 \rfloor},...,\\
&...\\
&x_0^{\lfloor d_i \rfloor}x_1^{p_i+1} \cdot ... \cdot x_{q_i}^{p_i+1}x_{q_i+1}^{p_i}\cdot ... \cdot x_{i-1}^{p_i},
x_0^{p_i+1} x_1^{\lfloor d_i \rfloor} x_2^{p_i+1} \cdot ... \cdot x_{q_i}^{p_i+1}x_{q_i+1}^{p_i}\cdot ... \cdot x_{i-1}^{p_i}, ...
\\
&...,\\
&x_0^{\lfloor d_{n+1} \rfloor}x_1^{p_{n+1}} \cdot ... \cdot x_{q_{n+1}}^{p_{n+1}+1}x_{q_{n+1}+1}^{p_{n+1}}\cdot ... \cdot x_{n}^{p_{n+1}},
x_0^{p_{n+1}+1} x_1^{\lfloor d_{n+1} \rfloor} x_2^{p_{n+1}+1} \cdot ... \cdot x_{q_{n+1}}^{p_{n+1}+1}x_{q_{n+1}+1}^{p_{n+1}}\cdot ... \cdot x_{n}^{p_{n+1}}, ...
\end{align*}

\end{construction}

Let $\Delta_i:=\lceil \frac{d-\lfloor d_{i} \rfloor}{i-1} \rceil $ and $\Delta_{\max}:=\max \limits_{2 \leq i \leq k} \{\Delta_i\}$. We have 
\begin{align*}
    \Delta_{\max}& \leq \max \limits_{i} \{\frac{d- d_{i} +1}{i-1} +1\}\\
    &=\frac{i}{i-1}-(A+1) d\\
    & \leq  2-(A+1)d
\end{align*}

\begin{lemma}
If $d\geq n^3+4 n^2-n$, then the syzygy bundle $E_{k,1}$ in Construction 1 satisfies \[\mu_\text{max}(E_{k,1})\leq \frac{4(n+1)}{(n^2+5n+2)}-\frac{n^2+5n+4}{n^2+5n+2} d.\]
\end{lemma}

\begin{proof}
    By \ref{Bre}, we need to show
    \begin{equation} \label{eq:1}
        \frac{d_J-rd}{r-1}\leq Ad
    \end{equation}
    for all possible choice of $J\subset I$ and for all $r\geq 2$ . It suffices to prove $d_J \leq  d_r $.  \par
    Since there are $R:=1+2+...+(n+1)=\frac{(n+1)(n+2)}{2}$ monomials containing $x_0$ in the construction, we have $d_J=0$ for $r \geq R +1$. Therefore, we only need to prove $d_J \leq  d_r $ for $2\leq r \leq R$. \par
    Given $2\leq r \leq R$, suppose we choose $J=\{f_1,...,f_r\}$ where $f_i$ is in $a_i$-th row of our construction and $a_1\leq ... \leq a_r$. To make inequality (\ref{eq:1}) true, we need $d_J \leq d_r$ for all $2 \leq r \leq R$.\par
    If $a_i=a_j$ for some $i$, $j$, then 
    \begin{align*}
        d_J -d_r& \leq d_J-d_R \\
        &\leq a_1 \Delta_{\max}-d_R \\
        &\leq (n+1)\Delta_{\max}-d_R \\
        &\leq (n+1)(2-(A+1)d)-d \left(A \left(\frac{1}{2} (n+1) (n+2)-1\right)+\frac{1}{2} (n+1) (n+2)\right).
    \end{align*}
    Thus we have $A \geq \frac{4(n+1)}{(n^2+5n+2)d}-\frac{n^2+5n+4}{n^2+5n+2}$. \par 
    If $a_1<...<a_r$, then
    \begin{align*}
        d_J -d_r&\leq \lfloor d_{a_r} \rfloor + (a_1-1)\Delta_{\max}-d_r \\
        &\leq d_{a_r} -(a_1-1)(A+1)d) + 2(a_1-1)-dr \\
        &= d_{a_r-a_1+1}-d_r + 2(a_1-1) \\
        &=2(a_1-1)+(a_r-a_1-r+1)(1+A)d \\
        & \leq 2((n+1-r+1)-1)+(1+A)d \\
        & \leq 2(n-1)+(1+A)d
    \end{align*}
    Thus $A\leq \frac{-d-2 n+2}{d}$.\par
    In conclusion, we need $\frac{4(n+1)}{(n^2+5n+2)d}-\frac{n^2+5n+4}{n^2+5n+2} \leq A \leq \frac{-d-2 n+2}{d}$. This is true when $d\geq n^3+4 n^2-n$. \par
    Let $W \subset E_{k,1}$ be a subbundle of rank $s$. By \ref{Bre}, $\mu(W) \leq \max_{|J|=s+1}\{ \frac{d_J-(s+1)d}{s}\}$. 
\end{proof}

We will use this $E_{k,1}$ to find a upper bound of a general kernel bundle $E_{a,b}$. To do this, we need the following  proposition.

\begin{proposition}
\label{Ext}
Let $E_{a_1,b_1}$ and  $E_{a_2,b_2}$ be kernel bundles on $\mathbb{P}^n_K$. Let $a>b$ be positive integers with $a_1+a_2=a$, $b_1+b_2=b$. There exists a kernel bundles $E_{a,b}$ such that there is a non-split extension 
\[0\xrightarrow{}E_{a_1,b_1}\xrightarrow{}E_{a,b}\xrightarrow[]{}E_{a_2,b_2}\xrightarrow{}0.\]
\end{proposition}

\begin{proof}
Suppose $E_{a_1,b_1}$ and  $E_{a_2,b_2}$ are defined by short exact sequences \[0\xrightarrow{}E_{a_1,b_1}\xrightarrow{}\mathcal{O}_{\mathbb{P}^n}(-d)^{a_1}\xrightarrow[]{\phi_1}\mathcal{O}_{\mathbb{P}^n}^{b_1}\xrightarrow{}0,\]
and   \[0\xrightarrow{}E_{a_2,b_2}\xrightarrow{}\mathcal{O}_{\mathbb{P}^n}(-d)^{a_1}\xrightarrow[]{\phi_2}\mathcal{O}_{\mathbb{P}^n}^{b_1}\xrightarrow{}0,\]
where $\phi_1$, $\phi_2$ are represented by matrices $M_1$, $M_2$ of polynomials of degree $d$. \par 
Let $N$ be a non-degenerate $a_1\times b_2$ matrix of polynomials of degree $d$. Let $M:=\begin{bmatrix}
M_1 & N\\
0 & M_2
\end{bmatrix}$. Then $M$ defines a surjective map $\phi:\mathcal{O}_{\mathbb{P}^n}(-d)^{a} \longrightarrow \mathcal{O}_{\mathbb{P}^n}^{b}$, which gives a kernel bundle $E_{a,b}$. By this construction, we have a non-split short exact sequence \[0\xrightarrow{}E_{a_1,b_1}\xrightarrow{}E_{a,b}\xrightarrow[]{}E_{a_2,b_2}\xrightarrow{}0.\]
\end{proof}

In the following theorem, we use the syzygy bundle $E_{k,1}$ in Construction 1 to find an upper bound of $\mu_{\max}(E_{a,b})$.  

\begin{theorem}
\label{j=1}
For $d\gg0$, a general kernel bundle $E_{mb-1,b}$ on $\mathbb{P}^n$ is semistable. Furthermore, we have the bound $\mu_{max}(E_{mb-1,b})\leq \frac{d (-n-1) (n+4)-4 (-n-1)}{n^2+5 n+2}$ for $d>\frac{6 b n^3+10 b n^2+10 b n+6 b-8 n-8}{b n^2-8 b n-b-4}$. For simplicity, we note $B=\frac{d (-n-1) (n+4)-4 (-n-1)}{n^2+5 n+2}$ .
\end{theorem}
\begin{proof}
 First, we prove the case when $m=k$. We prove every subsheaf $W \subsetneq E_{kb-1,b}$ satisfies $\mu(W) \leq B$  for $d\gg0$. This implies that $E_{kb-1,b}$ is stable because 
\begin{align*} 
\mu(E_{kb-1,b})-\mu(W) 
& =\frac{d (1-b k)}{b k-b-1}-\frac{d (-n-1) (n+4)-4 (-n-1)}{n^2+5 n+2} \\
&= \frac{-2 - 4 b + 2 b k - 5 b n - b n^2}{(-1 - b + b k) (2 + 5 n + n^2))} d +\frac{4 (-n-1)}{n^2+5 n+2}.
\end{align*} 
To prove this number is positive for $d\gg0$, it suffices to show $-2 - 4 b + 2 b k - 5 b n - b n^2>0$. \par 
Since $$-2 - 4 b + 2 b k - 5 b n - b n^2 > -2 - 4 b + 2 b k' - 5 b n - b n^2=-2 + \frac{1}{2} b (-1 - 8 n + n^2)$$
when $n>\sqrt{19}+4$, namely $n\geq 9$. For $n=3,4,5,6,7,8$, we can compute $k$ individually as $k=15, 21, 33, 41, 56, 69$. Correspondingly, we have $$\mu(E_{kb-1,b})-\mu(W)\geq -2 + 2 b, -2 + 2 b, -2 + 12 b, -2 + 12 b, -2 + 24 b, -2 + 30 b.$$ These are all positive numbers when $b\geq 2$. Therefore, we have $\mu_{max}(E_{mb-1,b})\leq B$ when $d>-\frac{4 (-1 - b + b k) (1 + n)}{-2 b k+b n^2+5 b n+4 b+2}$. \par
Now we prove $\mu(W)\leq B$ by induction on $b$. When $b=1$, by Lemma 1, $E_{k-1,1}$ satisfies $\mu_{max}(E_{k-1,1}) \leq \mu_{max}(E_{k,1})\leq B$. 
For a general $E_{kb-1,b}$, by Proposition \ref{Ext}, there is a kernel bundle $E_{k(b-1)-1,b-1}$ which fits in the following short exact sequence 
\[0\xrightarrow{}E_{k(b-1)-1,b-1}\xrightarrow{}E_{kb-1,b}\xrightarrow[]{\psi}E_{k,1}\xrightarrow{}0\]
where $E_{k,1}$ is the syzygy bundle constructed above. \par
Let $W \subsetneq E_{kb-1,b}$ and $W_1=\psi(W)$. \par
If $r(W_1)<r(E_{k,1})$, then $\mu(W)\leq \mu(W_1)\leq \mu_\text{max}(E_{k,1}) $. The theorem is proved in this case.  \par
If $r(W_1)=r(E_{k,1})=k-1$, then $r:=r(W)\geq k$. Let $W_2$ be the kernel of $W\xrightarrow{}W_1$. We have a short exact sequence 
\[0\xrightarrow{}W_2\xrightarrow{}W\xrightarrow[]{}W_1\xrightarrow{}0.\]
By induction hypothesis, $\mu(W_2) \leq \frac{4(n+1)}{(n^2+5n+2)}-\frac{n^2+5n+4}{n^2+5n+2} d$. 
Since $\text{deg}(W)=\text{deg}(W_2)+\text{deg}(W_1)$, we get 
\begin{align*} 
\mu(W) & =\frac{\mu(W_2)r(W_2)+\mu(W_1)r(W_1)}{r(W)}
= \frac{(r-k+1)\mu(W_2)-kd}{r}
\end{align*}
This number increases as $r$ increases when $\mu(W_2)<-\frac{k}{k-1}d$. Since $W_2$ is a proper subbundle of $E_{k(b-1)-1,b-1}$, by the induction hypothesis, 
\begin{align*}
    \mu(W_2) &\leq \frac{d (-n-1) (n+4)-4 (-n-1)}{n^2+5 n+2} 
     < \mu(E_{kb-1,1}) 
    = \frac{d (1-b k)}{b k-b-1}
    <-\frac{k}{k-1}d.
\end{align*}
Picking $r=k$, we get 
\[ \mu(W) \leq \frac{\mu(W_2)-kd}{k} \leq \mu(W_2) \leq B.\]
For $2 \leq m < k$, we drop the monomials in the construction of $E_{k,1}$ to make it a syzygy bundle $E_{m,1}$ with  $\mu (E_{m,1})\leq B$. Then the same argument implies $\mu_{\text{max}}(E_{mb-1,b})\leq B$ for $d\gg0$.
\end{proof}

\begin{theorem}
\label{main}
    For a given pair of positive integers $(a,b)$, if we can write $a=mb-j$ for some integers $j$, $m$ with $0\leq j \leq b-1$ and $2 \leq m \leq k$, then a general kernel bundle $E_{a,b}$ on $\mathbb{P}^n$ is semi-stable for $d\gg0$. 
\end{theorem}

\begin{proof}
    We will show $\mu_{\text{max}}(E_{mb-j,b}) \leq B$ for $d\gg0$. We induct on $j$. \par 
    For $j=1$, this is true by Theorem 4. \par
    Write $b=sj+l$ for some $0 \leq l \leq j-1$. Then $E_{a,b}=E_{m(sj+l)-j,sj+l}$. If $l=0$, $E_{a,b}$ is semi-stable since it is a direct sum of stable bundles of the same slope.\par
    By Proposition \ref{Ext}, consider the short exact sequence of kernel bundles 
    $$0\xrightarrow{}E_{ms-1,s}\xrightarrow{}E_{m(sj+l)-j,sj+l}\xrightarrow[]{\psi}E_{m(s(j-1)+l)-(j-1),s(j-1)+l}\xrightarrow{}0. $$
    Let $W \subsetneq E_{a,b}$ and  $W_1=\psi(W)$. \par
    If $r(W_1)< r(E_{m(s(j-1)+l)-(j-1),s(j-1)+l})$, then $$\mu(W)\leq \mu(W_1) \leq \mu(E_{m(s(j-1)+l)-(j-1),s(j-1)+l}).$$ By induction hypothesis, $E_{a,b}$ is semi-stable. \par
    If $r(W_1) = r(E_{m(s(j-1)+l)-(j-1),s(j-1)+l})=m ((j-1) s+l)-(j-1) s-j-l+1$, let $W_2$ be the kernel of $W\xrightarrow{}W_1$. We have a short exact sequence 
\[0\xrightarrow{}W_2\xrightarrow{}W\xrightarrow[]{}W_1\xrightarrow{}0.\]
By the induction hypothesis, $\mu(W_2) \leq B$. Write $r$ for $r(W)$. 
    Since $\text{deg}(W)=\text{deg}(W_2)+\text{deg}(W_1)$, we get 
\begin{align*} 
\mu(W) & =\frac{\mu(W_2)r(W_2)+\mu(W_1)r(W_1)}{r(W)}\\
&= \frac{(r-(m ((j-1) s+l)-(j-1) s-j-l+1))\mu(W_2)-d (-m ((j-1) s+l)+j-1)}{r}\\
&\leq \frac{(r-(m ((j-1) s+l)-(j-1) s-j-l+1))B-d (-m ((j-1) s+l)+j-1)}{r}. 
\end{align*}
    This number increases as r increase. Setting $r=m ((j-1) s+l)-(j-1) s-j-l+2$, we have
    \begin{align*}
        \mu(W) &\leq \frac{B-d (-m ((j-1) s+l)+j-1)}{m ((j-1) s+l)-(j-1) s-j-l+2}\\
        &\leq \frac{d \left(m ((j-1) s+l)-j+\frac{(-n-1) (n+4)}{n^2+5 n+2}+1\right)}{j ((m-1) s-1)+l (m-1)-m s+s+2}\\
        &-\frac{4 (-n-1)}{\left(n^2+5 n+2\right) (j ((m-1) s-1)+l (m-1)-m s+s+2)}
    \end{align*}\\
    Thus 
    \begin{multline}
        \mu(E_{a,b})-\mu(W)= d \left(\frac{-m ((j-1) s+l)+j-\frac{(-n-1) (n+4)}{n^2+5 n+2}-1}{j ((m-1) s-1)+l (m-1)-m s+s+2}+\frac{j-m (l+\text{sj})}{-j+m (l+\text{sj})-l-\text{sj}}\right)+\\ \frac{4 (-n-1)}{\left(n^2+5 n+2\right) (j ((m-1) s-1)+l (m-1)-m s+s+2)}
    \end{multline}
    This number is positive when $d\gg0$. Therefore, $E_{a,b}$ is semi-stable when $(a,b)$ satisfies the condition in the theorem. 
\end{proof}

Theorem \ref{j=1} and Theorem \ref{main} do not cover all possible pairs of a and b. For example, on $\mathbb{P}^2$, these theorems do not 
 show the stability of the kernel bundle $E_{17,2}$. \par
The following proposition provide a new method prove the stability of kernel bundles. In \cite{Co10}, the authors provide a way to construct syzygy bundles with small maximal slopes in Chapter 3. Using their construction of $E_{8,1}$ and $E_{9,1}$, we can find a bound of $\mu_{max}(E_{17,2})$ and show this bundle is stable. 
\begin{proposition}
\label{Ex}
    On $\mathbb{P}^2$, a general kernel bundle $E_{17,2}$ is stable for $d\gg0$. 
\end{proposition} 
\begin{proof}
    Let $e_0,e_1,e_2$ be the integers satisfying $e_0=\lceil \frac{d}{3} \rceil$, $e_0\geq e_1 \geq e_2$ and $e_0-e_2 \leq 1$. Let $E_{8,1}$ be the syzygy bundle defined by monomials $$x_0^d,x_1^d,x_2^d, x_0^{e_0}x_1^{e_1}x_2^{e_2},x_0^{e_2}x_2^{e_0+e_1},x_1^{e_0+e_1}x_2^{e_2},x_1^{e_0}x_2^{e_1+e_2}.$$ By \ref{Bre}, we know $\mu_{\max}(E_{8,1})=\max_{J\subset I,r\geq 2}\{ \frac{d_J-rd}{r-1}\}$. For each given $r$, we compute the largest possible slope $\mu_{\max}$ of subbundle of rank $r'=r-1$ in the following table. 
\begin{center}
\begin{tabular}{| c| c| c| c|c|c|}
\hline
 $r'$ & 1 & 2 &3 &$\geq 4$   \\ 
 \hline
 $\mu_{\max}$ &$-\frac{4}{3}d+O(1)$ &$-\frac{4}{3}d+O(1)$ &$-\frac{11}{9}d+O(1)$ &$-\frac{r'+1}{r'}d+O(1)$  \\  
 \hline
\end{tabular}
\end{center}
Thus, $\mu_{\text{max}}(E_{8,1})=-\frac{7}{6}d+O(1)$. It is achieved when $r'=7$. \par
Now we construct the syzygy bundle $E_{9,1}$. Let $d=3m+t$, $0 \leq t <3$, $i_l:=lm+\min(l,t)$,  $l=1,2$, and $E_{9,1}$ be the syzygy bundle given by monomials $$x_0^d,x_1^d,x_2^d,x_0^{i_1}x_1^{d-i_1},x_0^{i_2}x_1^{d-i_2},x_0^{d-i_1}x_2^{i_1},x_0^{d-i_2}x_2^{i_2},x_1^{i_1}x_2^{d-i_1},x_1^{i_2}x_2^{d-i_2}.$$
We compute $\mu_{\max}(E_{9,1})$ in the following table. 
\begin{center}
\begin{tabular}{|c| c| c| c|c|c|c|}
\hline
$r'$ & 1 & 2 &3 &4 &$\geq 5$   \\ 
 \hline
 $\mu_{\max}$ &$-\frac{4}{3}d+O(1)$ &$-\frac{7}{6}d+O(1)$ &$-\frac{11}{9}d+O(1)$ &$-\frac{7}{6}d+O(1)$ &$-\frac{r'+1}{r'}d+O(1)$  \\  
 \hline
\end{tabular}
\end{center}
Thus, $\mu_{\text{max}}(E_{9,1})=-\frac{8}{7}d+O(1)$. It is achieved when $r'=8$. \par
Consider the bundle $E_{17,2}$ constructed as the extension $E_{a,b}$ in \ref{Ext}. We get short exact sequence 
$$0\xrightarrow{}E_{8,1}\xrightarrow{}E_{17,2}\xrightarrow[]{}E_{9,1}\xrightarrow{}0.$$
Let $W \subsetneq E_{17,2}$ and $W_1=\psi(W)$.  \par
If $r(W_1)<r(E_{9,1})$, then $\mu(W)\leq \mu(W_1)\leq \mu_\text{max}(E_{9,1})=-\frac{8}{7}d < \mu(E_{17,2}) $. The proposition is proved. \par
If $r(W_1)=r(E_{9,1})=8$, then $r':=r(W)\geq 9$. Let $W_2$ be the kernel of $W\xrightarrow{}W_1$. We have a short exact sequence 
\[0\xrightarrow{}W_2\xrightarrow{}W\xrightarrow[]{}W_1\xrightarrow{}0.\]
Since $\text{deg}(W)=\text{deg}(W_2)+\text{deg}(W_1)$, we get 
\begin{align*} 
\mu(W) & =\frac{\mu(W_2)r(W_2)+\mu(W_1)r(W_1)}{r(W)}
= \frac{(r'-8)\mu(W_2)-9d}{r'}.
\end{align*}
Here $W_2$ is a subbundle of $E_{8,1}$. According to the maximal slope numbers we compute above, we conclude that $\mu(W) \leq -\frac{8}{7}d +O(1)$. Thus, $\mu(W) < \mu(E_{17,2})$. $E_{17,2}$ is stable. 
\end{proof}\par

Note that Proposition \ref{Ex} is not covered by our main theorem \ref{main}. We expect this method works for more bundles in the form of $E_{a,2}$. However, for bundles of the form $E_{a,b}$ with $b\geq 3$, the construction in \cite{Co10} used in this proposition is not effective. The main difficulty is the explicit construction of a syzygy bundle with smallest possible maximal slope.

\nocite{*}
\bibliographystyle{plain}
\bibliography{stability}

\end{document}